\documentclass[a4,12pt]{amsart}
\usepackage{amsfonts}
\usepackage{amssymb}
\usepackage{amsmath}
\usepackage{amsthm}
\usepackage{fullpage}

\theoremstyle{plain}
\newtheorem{thm}{Theorem}[section]
\newtheorem{lem}[thm]{Lemma}
\theoremstyle{definition}

\newtheorem{ex}[thm]{Example}
\newtheorem{rem}[thm]{Remark}

\numberwithin{equation}{section}

\begin{document}
\title[Hammerstein systems with first derivative dependence]{Nontrivial solutions of systems of Hammerstein integral equations with first derivative dependence}
\author[G. Infante]{Gennaro Infante}
\address{Dipartimento di Matematica e Informatica, Universit\`{a} della
Calabria, 87036 Arcavacata di Rende, Cosenza, Italy}
\email{gennaro.infante@unical.it}
\author[F. Minh\'{o}s]{Feliz Minh\'{o}s}
\address{Departamento de Matem\'{a}tica, Escola de Ci\^{e}ncias e Tecnologia%
\\
Centro de Investiga\c{c}\~{a}o em Matem\'{a}tica e Aplica\c{c}\~{o}es (CIMA)%
\\
\ Instituto de Investiga\c{c}\~{a}o e Forma\c{c}\~{a}o Avan\c{c}ada \\
\ Universidade de \'{E}vora. Rua Rom\~{a}o Ramalho, 59 \\
\ 7000-671 \'{E}vora, Portugal}
\email{fminhos@uevora.pt}

\begin{abstract}
By means of classical fixed point index, we prove new results on the existence, non-existence, localization and multiplicity of nontrivial solutions for systems of Hammerstein integral equations where the nonlinearities are allowed to depend on the first derivative. As a byproduct of our theory we discuss the existence of positive solutions of a system of third order ODEs subject to nonlocal boundary conditions.
Some examples are provided in order to illustrate the applicability of the theoretical results.
\end{abstract}

\maketitle

\section{Introduction}
Motivated by earlier work of do {\'{O}}, Lorca and Ubilla~\cite{dolo} on radial solutions of elliptic systems, Infante and Pietramala~\cite{gipp-nodea} studied the existence, multiplicity and non-existence of \emph{nontrivial} solutions of systems of Hammerstein integral equations of the type
\begin{equation*}\label{ip-sys}
 \left\{
\begin{array}{c}
u(t)=\int_{0}^{1}k_{1}(t,s)g_{1}(s)f_{1}(s, u(s), v(s))\,ds, \\ 
v(t)=\int_{0}^{1}k_{2}(t,s)g_{2}(s)f_{2}(s, u(s), v(s))\,ds.%
\end{array}
 \right.
\end{equation*}%
The methodology of~\cite{gipp-nodea} is based on classical fixed point index theory and the authors work in a suitable cone in $\left( C[0,1]\right) ^{2}$. Due to the choice of the space involved, the setting of~\cite{gipp-nodea} does not allow derivative dependence in the nonlinearities. 

On the other hand, Minh\'{o}s and de Sousa~\cite{fm+rs} studied the system of third order ordinary differential equations subject to nonlocal boundary conditions
\begin{equation}\label{FelRob}
\left\{ 
\begin{array}{c}
-u^{\prime \prime \prime }(t)=f_{1}(t, v(t), v^{\prime }(t)), \\ 
-v^{\prime \prime \prime }(t)=f_{2}(t, u(t), u^{\prime }(t)), \\ 
u(0)=u^{\prime }(0)=0,u^{\prime }(1)=\alpha  u^{\prime }(\eta ), \\ 
v(0)=v^{\prime }(0)=0,v^{\prime }(1)=\alpha  v^{\prime }(\eta ),%
\end{array}%
\right.  
\end{equation}%
where $0<\eta<1$ and $1<\alpha<{1}/{\eta}$.
The approach of~\cite{fm+rs} relies on the celebrated Krasnosel'ski\u\i{}-Guo fixed point theorem and on the rewriting the system~\eqref{FelRob} in the form
\begin{equation}\label{FelRob-intsys}
 \left\{
\begin{array}{c}
u(t)=\int_{0}^{1}k(t,s)f_{1}(s, v(s),v^{\prime }(s))\,ds, \\ 
v(t)=\int_{0}^{1}k(t,s)f_{2}(s,u(s),u^{\prime
}(s))\,ds.%
\end{array}
 \right.
\end{equation}%
Minh\'{o}s and de Sousa proved the existence of \emph{one} positive solution of the system~\eqref{FelRob-intsys}, by assuming suitable superlinear/sublinear behaviours of the nonlinearities. A key ingredient in~\cite{fm+rs} is the use of the cone 
\begin{equation}\label{FelRob-cone}
\hat{K}:=\left\{ w\in C^{1}[0,\,1]: w(t)\geq 0, \underset{t\in [ \frac{\eta}{\alpha},\eta]}%
{\min }w(t)\geq c\Vert w\Vert _{C},\underset{t\in [\frac{\eta}{\alpha},\eta]}{\min }w^{\prime }(t)\geq d\Vert w^{\prime }\Vert
_{C}\right\} ,  
\end{equation}%
where $c, d\in (0,1]$ and
$\Vert w\Vert _{C}:=\underset{t\in \lbrack 0,\,1]\,%
}{\max }|w(t)|$. The cone~\eqref{FelRob-cone} is similar to 
a cone of \emph{non-negative} functions first used by Krasnosel'ski\u\i{}, see e.g. \cite{krzab}, and D.~Guo, see e.g. \cite{guolak} in the space $C[0,\,1]$. Note that the functions in~\eqref{FelRob-cone} are non-negative and their derivatives are non-negative on a subset of $[0,1]$.

Here we make use of a new cone of functions that are allowed to \emph{change sign}, similar to one introduced, in the space of  continuous functions, by Infante and Webb~\cite{gijwjiea}. With this ingredient we prove existence, multiplicity and non-existence results for \emph{nontrivial} solutions of the systems of integral equations of the kind
\begin{equation*}\label{syst-intro}
 \left\{
\begin{array}{c}
u(t)=\int_{0}^{1}k_{1}(t,s)g_{1}(s)f_{1}(s,u(s),u^{\prime
}(s),v(s),v^{\prime }(s))\,ds, \\ 
v(t)=\int_{0}^{1}k_{2}(t,s)g_{2}(s)f_{2}(s,u(s),u^{\prime
}(s),v(s),v^{\prime }(s))\,ds,%
\end{array}
 \right.
\end{equation*}%
extending the results of~\cite{gipp-nodea} to this different setting. 

We note that our approach can be also used to prove the existence of \emph{non-negative} solutions; we highlight this fact by considering a generalization of the system~\eqref{FelRob}, that is
\begin{equation}\label{bvp-ex-intro}
\left\{ 
\begin{array}{c}
-u^{\prime \prime \prime }(t)=g_{1}(t)f_{1}(t,\,u(t),\,u^{\prime
}(t),v(t),\,v^{\prime }(t)), \\ 
-v^{\prime \prime \prime }(t)=g_{2}(t)f_{2}(t,\,u(t),\,u^{\prime
}(t),v(t),\,v^{\prime }(t)), \\ 
u(0)=u^{\prime }(0)=0,u^{\prime }(1)=\alpha_{1}  u^{\prime }(\eta_{1} ), \\ 
v(0)=v^{\prime }(0)=0,v^{\prime }(1)=\alpha_{2} v^{\prime }(\eta_{2} ),%
\end{array}%
\right.  
\end{equation}%
where $0<\eta_{i} <1$,
$1<\alpha_{i}  <\frac{1}{\eta_{i} }$. Note that the boundary conditions in~\eqref{bvp-ex-intro} can generate two different kernels and the nonlinearities are allowed to have a stronger coupling with respect to the ones present in~\eqref{FelRob}.

Some examples are given to show that the constants that occur in our theoretical results can be computed.

\section{The system of integral equations}
We begin by stating some assumptions on the terms that occur in
the system of Hammerstein integral equations
\begin{equation}
 \left\{
\begin{array}{c}
u(t)=\int_{0}^{1}k_{1}(t,s)g_{1}(s)f_{1}(s,u(s),u^{\prime
}(s),v(s),v^{\prime }(s))\,ds, \\ 
v(t)=\int_{0}^{1}k_{2}(t,s)g_{2}(s)f_{2}(s,u(s),u^{\prime
}(s),v(s),v^{\prime }(s))\,ds,%
\end{array}
\label{syst}
 \right.
\end{equation}%
namely:
\begin{itemize}
\item[$(A1)$] For $i=1,2$, $f_{i}:[0,1]\times 
\mathbb{R}
^{4}\rightarrow \lbrack 0, +\infty )$ is a $L^{\infty }$-Carath\'{e}odory
function, that is, $f_{i}(\cdot ,u_{1},u_{2},v_{1},v_{2})$ is measurable
for each fixed $(u_{1},u_{2},v_{1},v_{2}),$ $f_{i}(t,\cdot ,\cdot ,\cdot
,\cdot )$ is continuous for almost every (a.e.) $t\in \lbrack 0,1]$, and
for each $r>0$ there exists $\varphi _{i,r}\in L^{\infty }[0,1]$ such that{} 
\begin{equation*}
f_{i}(t,u_{1},u_{2},v_{1},v_{2})\leq \varphi _{i,r}(t)\;\text{ for }%
\;u_{1},u_{2},v_{1},v_{2}\in \lbrack -r,r]\;\text{ and a.\thinspace e.}%
\;t\in \lbrack 0,1].
\end{equation*}%
{}

\item[$(A2)$] For every $i=1,2$, $k_{i}:[0,1]^{2}\rightarrow 
\mathbb{R}
$ is such that $k_{i} $ are measurable, and for
all $\tau \in \lbrack 0,1]$ we have 
\begin{equation*}
\lim_{t\rightarrow \tau }|k_{i}(t,s)-k_{i}(\tau ,s)|=0,\;\text{ for
a.\thinspace e.}\,s\in \lbrack 0,1]
\end{equation*}%
and%
\begin{equation*}
\lim_{t\rightarrow \tau }\left\vert \frac{\partial k_{i}}{\partial t}(t,s)-%
\frac{\partial k_{i}}{\partial t}(\tau ,s)\right\vert =0,\;\text{ for
a.\thinspace e.}\,s\in \lbrack 0,1].
\end{equation*}%
{}

\item[$(A3)$] For every $i=1,2$, there exist subintervals $[a_{i},b_{i}],[%
\gamma _{i},\delta_{i}]\subseteq \lbrack 0,1]$, functions $\phi
_{i},\psi _{i}\in L^{\infty }[0,1]$, and constants $c_{i},d_{i}\in (0,1]$,
such that%
\begin{eqnarray*}
\left\vert k_{i}(t,s)\right\vert &\leq &\phi _{i}(s)\text{ for }t\in \lbrack
0,1]\text{ and a.\thinspace e.}\,s\in \lbrack 0,1], \\
\left\vert \frac{\partial k_{i}}{\partial t}(t,s)\right\vert &\leq &\psi
_{i}(s)\text{ for }t\in \lbrack 0,1]\text{ and a.\thinspace e.}\,s\in
\lbrack 0,1] \\
k_{i}(t,s) &\geq &c_{i}\phi _{i}(s)\text{ for }t\in \lbrack a_{i},b_{i}]%
\text{ and a.\thinspace e.}\,s\in \lbrack 0,1]. \\
\frac{\partial k_{i}}{\partial t}(t,s) &\geq &d_{i}\psi _{i}(s)\text{ for }%
t\in \lbrack \gamma _{i},\delta_{i}]\text{ and a.\thinspace e.}\,s\in
\lbrack 0,1].
\end{eqnarray*}%
{}

\item[$(A4)$] For every $i=1,2$, we have $g_{i}\in
L^{1}[0,1]$, $g_{i}(t)\geq 0$ a.e. $t\in \lbrack 0,1]$, 
$\int_{a_{i}}^{b_{i}}\phi _{i}(s)g_{i}(s)\,ds>0$ and $\int_{\gamma
_{i}}^{\delta_{i}}\psi _{i}(s)g_{i}(s)\,ds>0$. {}
\end{itemize}

Forward in the paper we use the space $\left( C^{1}[0,1]\right) ^{2}$ equipped with
the norm 
\begin{equation*}
\Vert (u,v)\Vert :=\max \{\Vert u\Vert _{C^{1}},\Vert v\Vert _{C^{1}}\},
\end{equation*}%
where $\Vert w\Vert _{C^{1}}:=\max \left\{ \Vert w\Vert _{C},\Vert w^{\prime
}\Vert _{C}\right\}$.

For the reader's convenience, we recall that a \emph{cone} $K$ in a Banach
space $X$ is a closed convex set such that $\lambda \,x\in K$ for $x\in K$
and $\lambda \geq 0$ and $K\cap (-K)=\{0\}$.

Consider, in the space $C^{1}[0,\,1]$, the cones
\begin{equation}
\tilde{K_{i}}:=\left\{ w\in C^{1}[0,\,1]:\underset{t\in \lbrack a_{i},b_{i}]}%
{\min }w(t)\geq c_{i}\Vert w\Vert _{C},\underset{t\in \lbrack \gamma_{i}, \delta _{i}]}{\min }w^{\prime }(t)\geq d_{i}\Vert w^{\prime }\Vert
_{C}\right\} ,  \label{Ktil}
\end{equation}%
and their product in $\left( C^{1}[0,1]\right) ^{2}$ defined by 
\begin{equation}\label{cone-sign}
\begin{array}{c}
K:=\{(u,v)\in \tilde{K_{1}}\times \tilde{K_{2}}\}.%
\end{array}%
\end{equation}

By a \emph{nontrivial} solution of the system (\ref{syst}) we mean a
solution $(u,v)\in K$ of (\ref{syst}) such that $\Vert (u,v)\Vert \neq 0$.
Note that the functions in $\tilde{K_{i}}$ are non-negative on the sub-intervals 
$[a_{i},b_{i}]$ and non-decreasing on $[\gamma _{i},\delta_{i}]$ but,
nevertheless, they can change sign or have a different variation in $[0,1]$.

We define the integral operator 
\begin{equation}
\begin{array}{c}
T(u,v)(t):=\left( 
\begin{array}{c}
T_{1}(u,v)(t) \\ 
T_{2}(u,v)(t)%
\end{array}%
\right) 
=\left( 
\begin{array}{c}
\int_{0}^{1}k_{1}(t,s)g_{1}(s)f_{1}(s,u(s),u^{\prime }(s),v(s),v^{\prime
}(s))\,ds \\ 
\int_{0}^{1}k_{2}(t,s)g_{2}(s)f_{2}(s,u(s),u^{\prime }(s),v(s),v^{\prime
}(s))\,ds%
\end{array}%
\right),%
\end{array}
\label{opT}
\end{equation}%
and prove that $T$ leaves the cone $K$ invariant and is compact.

\begin{lem}
The operator $T$ given by \eqref{opT} maps $K$ into $K$ and
is compact.
\end{lem}

\begin{proof}
Take $(u,v)\in K.$ Then, by (A3), 
\begin{equation*}
\Vert T_{1}(u,v)\Vert _{C}\leq \int_{0}^{1}\phi
_{1}(s)g_{1}(s)f_{1}(s,u(s),u^{\prime }(s),v(s),v^{\prime }(s))\,ds,
\end{equation*}%
and%
\begin{align*}
\min_{t\in \lbrack a_{1},b_{1}]}T_{1}(u,v)(t) 
&= \min_{t\in \lbrack a_{1},b_{1}]} \int_{0}^{1}k_{1}(t,s)g_{1}(s)f_{1}(s,u(s),u^{\prime }(s),v(s),v^{\prime }(s))\,ds \\
&\geq c_{1}\int_{0}^{1}\phi
_{1}(s)g_{1}(s)f_{1}(s,u(s),u^{\prime }(s),v(s),v^{\prime }(s))\,ds \\
&\geq c_{1}\Vert T_{1}(u,v)\Vert _{C}.
\end{align*}%
Moreover%
\begin{equation*}
\Vert \left( T_{1}(u,v)\right) ^{\prime }\Vert _{C}\leq \int_{0}^{1}\psi
_{1}(s)g_{1}(s)f_{1}(s,u(s),u^{\prime }(s),v(s),v^{\prime }(s))\,ds,
\end{equation*}%
and%
\begin{eqnarray*}
\min_{t\in \lbrack \gamma _{1},\delta_{1}]}\left( T_{1}(u,v)(t)\right)
^{\prime } &=&\min_{t\in \lbrack \gamma _{1},\delta_{1}]}\int_{0}^{1}\frac{%
\partial k_{1}}{\partial t}(t,s)g_{1}(s)f_{1}(s,u(s),u^{\prime
}(s),v(s),v^{\prime }(s))\,ds \\
&\geq &d_{1}\int_{0}^{1}\psi _{1}(s)g_{1}(s)f_{1}(s,u(s),u^{\prime
}(s),v(s),v^{\prime }(s))\,ds \\
&\geq &d_{1}\Vert \left( T_{1}(u,v)\right) ^{\prime }\Vert _{C}.
\end{eqnarray*}%
Therefore $T_{1}\tilde{K_{1}}\subset \tilde{K_{1}}$. By similar arguments it can be proved that $T_{2}\tilde{K_{2}}\subset \tilde{%
K_{2}}$. 

The compactness of $T$ follows, in a routine way, by the Ascoli-Arzel\`{a} Theorem.
\end{proof}

To specify our notation, for $\Omega $ an open bounded subset with $\Omega
\subset K$ (endowed with the relative topology), we denote by $\overline{%
\Omega }$ and $\partial \Omega $ the closure and the boundary relative to $%
K, $ respectively. If $\Omega $ is an open bounded subset of $X$ then we
write $\Omega _{K}=\Omega \cap K$, an open subset of $K$.

The next Lemma summarizes some classical results on fixed point index (more
details can be seen in the books~\cite{Amann-rev, guolak}).

\begin{lem}
\label{index_lem}Let $\Omega $ be an open bounded set with $0\in \Omega _{K}$
and $\overline{\Omega }_{K}\neq K$. Assume that $F:\overline{\Omega }%
_{K}\rightarrow K$ is a compact map such that $x\neq Fx$ for all $x\in
\partial \Omega _{K}$. Then the fixed point index $i_{K}(F,\Omega _{K})$ has
the following properties.

\begin{itemize}
\item[(1)] If there exists $e\in K\setminus \{0\}$ such that $x\neq
Fx+\lambda e$ for all $x\in \partial \Omega _{K}$ and all $\lambda >0$, then 
$i_{K}(F,\Omega _{K})=0$.

\item[(2)] If $\mu x\neq Fx$ for all $x\in \partial \Omega _{K}$ and for
every $\mu \geq 1$, then $i_{K}(F,\Omega _{K})=1$.

\item[(3)] If $i_{K}(F,\Omega _{K})\neq 0$, then $F$ has a fixed point in $
\Omega _{K}$.

\item[(4)] Let $\Omega ^{1}$ be open in $X$ with $ \overline{
\Omega _{K}^{1}}
\subset \Omega _{K}$. If $i_{K}(F,\Omega _{K})=1$ and $i_{K}(F,\Omega
_{K}^{1})=0$, then $F$ has a fixed point in $\Omega _{K}\setminus \overline{%
\Omega _{K}^{1}}$. The same result holds if $i_{K}(F,\Omega _{K})=0$ and $%
i_{K}(F,\Omega _{K}^{1})=1$.
\end{itemize}
\end{lem}

Along the paper, we use the following (relative) open bounded sets in $K$: 
\begin{equation}
K_{\rho _{1},\rho _{2}}=\{(u,v)\in K:\Vert u\Vert _{C^{1}}<\rho _{1}\ \text{%
and}\ \Vert v\Vert _{C^{1}}<\rho _{2}\},  \label{Krho12}
\end{equation}%

For our index calculations we make use of the following Lemma, similar to
Lemma $5$ of~\cite{df-gi-do}. The novelty here is that we take into account
the derivative. We omit the simple proof. 

\begin{lem}
For the set defined by (\ref{Krho12}) we have that $%
(w_{1},w_{2})\in \partial K_{\rho _{1},\rho _{2}}$ \ iff \ $(w_{1},w_{2})\in
K$ and, for $i=1,2,$ 
\begin{equation*}
\max_{t\in \lbrack 0,1]}w_{1}(t)=\rho _{1},\, -\rho _{1}\leq
w_{1}^{\prime }(t)\leq \rho _{1},\, -\rho _{2}\leq w_{2}(t)\leq \rho
_{2},\, -\rho _{2}\leq w_{2}^{\prime }(t)\leq \rho _{2},
\end{equation*}%
or%
\begin{equation*}
-\rho _{1}\leq w_{1}(t)\leq \rho _{1},\, \max_{t\in \lbrack
0,1]}w_{1}^{\prime }(t)=\rho _{1},\, -\rho _{2}\leq w_{2}(t)\leq \rho
_{2},\, -\rho _{2}\leq w_{2}^{\prime }(t)\leq \rho _{2},
\end{equation*}%
or%
\begin{equation*}
-\rho _{1}\leq w_{1}(t)\leq \rho _{1},\, -\rho _{1}\leq w_{1}^{\prime
}(t)\leq \rho _{1},\, \max_{t\in \lbrack 0,1]}w_{2}(t)=\rho _{2},\text{ 
}-\rho _{2}\leq w_{2}^{\prime }(t)\leq \rho _{2},
\end{equation*}%
or%
\begin{equation*}
-\rho _{1}\leq w_{1}(t)\leq \rho _{1},\, -\rho _{1}\leq w_{1}^{\prime
}(t)\leq \rho _{1},\, -\rho _{2}\leq w_{2}(t)\leq \rho _{2},\, %
\max_{t\in \lbrack 0,1]}w_{2}^{\prime }(t)=\rho _{2}.
\end{equation*}
\end{lem}

\section{Existence results and non-existence results}

The existence results are obtained via the fixed point index on the set 
$K_{\rho _{1},\rho _{2}}$ given by~\eqref{Krho12}.
Firstly we obtain sufficient conditions for the fixed point index on the set $
K_{\rho _{1},\rho _{2}}$ to be 1.

\begin{lem}
\label{ind1b} Assume that
\begin{itemize}
\item[$(\mathrm{I}_{\rho _{1},\rho_2}^{1})$] 
 there exist $\rho _{1},\rho _{2}>0$ such that for every $%
i=1,2,$ 
\begin{equation}
f_{i}^{\rho _{1},\rho _{2}}<\min \left\{ m_{i},m_{i}^{\ast }\right\} 
\label{eqmestt}
\end{equation}%
{} where 
\begin{equation}
f_{i}^{\rho _{1},\rho _{2}}:=\sup \left\{ \frac{%
f_{i}(t,u_{1},u_{2},v_{1},v_{2})}{\rho _{i}}:(t,u_{1},u_{2},v_{1},v_{2})\in
\lbrack 0,1]\times \lbrack -\rho _{1},\rho _{1}]^{2}\times \lbrack -\rho
_{2},\rho _{2}]^{2}\right\} ,  \label{f^rho12}
\end{equation}%
\begin{equation}
\frac{1}{m_{i}}:=\max_{t\in \lbrack 0,1]}\int_{0}^{1}|k_{i}(t,s)|g_{i}(s)\,ds
\label{mi}
\end{equation}%
and{}%
\begin{equation}
\frac{1}{m_{i}^{\ast }}:=\max_{t\in \lbrack 0,1]}\int_{0}^{1}\left\vert 
\frac{\partial k_{i}}{\partial t}(t,s)\right\vert g_{i}(s)\,ds  \label{masti}
\end{equation}
\end{itemize}
Then $i_{K}(T,K_{\rho _{1},\rho _{2}})=1$.
\end{lem}

\begin{proof}
We claim that $\lambda (u,v)\neq T(u,v)$ for every $(u,v)\in \partial
K_{\rho _{1},\rho _{2}}$ and for every $\lambda \geq 1$, which implies that
the index is 1 on $K_{\rho _{1},\rho _{2}},$ by Lemma~\ref{index_lem}~(3).

Assume this is not true. Then there exist $\lambda \geq 1$ and $(u,v)\in
\partial K_{\rho _{1},\rho _{2}}$ such that $\lambda (u,v)=T(u,v)$.

Consider that 
\begin{equation}
\Vert u\Vert _{C}=\rho _{1},\Vert u^{\prime }\Vert _{C}\leq \rho _{1},\Vert
v\Vert _{C}\leq \rho _{2}\text{ and }\Vert v^{\prime }\Vert _{C}\leq \rho
_{2}  \label{u=rho1}
\end{equation}%
holds. Then we have
\begin{equation*}
\lambda \left\vert u(t)\right\vert \leq \int_{0}^{1}\left\vert
k_{1}(t,s)\right\vert g_{1}(s)f_{1}(s,u(s),u^{\prime }(s),v(s),v^{\prime
}(s))\,ds,
\end{equation*}%
and, taking the maximum over $[0,1],$ by (\ref{f^rho12}) and (\ref{mi})%
\begin{eqnarray*}
\lambda {\rho _{1}} &\leq &\max_{t\in \lbrack 0,1]}\int_{0}^{1}\left\vert
k_{1}(t,s)\right\vert g_{1}(s)f_{1}(s,u(s),u^{\prime }(s),v(s),v^{\prime
}(s))\,ds \\
&\leq &\max_{t\in \lbrack 0,1]}\int_{0}^{1}\left\vert k_{1}(t,s)\right\vert
g_{1}(s){\rho _{1}}f_{1}^{\rho _{1},\rho _{2}}\,ds \\
&\leq &{\rho _{1}}f_{1}^{\rho _{1},\rho _{2}}\dfrac{1}{m_{1}}.
\end{eqnarray*}%
By \eqref{eqmestt}, $\lambda \rho _{1}<\rho _{1},$ which contradicts the
fact that $\lambda \geq 1$.

If
\begin{equation*}
\Vert u\Vert _{C}\leq \rho _{1},\Vert u^{\prime }\Vert _{C}=\rho _{1},\Vert
v\Vert _{C}\leq \rho _{2}\text{ and }\Vert v^{\prime }\Vert _{C}\leq \rho
_{2},
\end{equation*}%
then we have
\begin{equation*}
\lambda \left\vert u^{\prime }(t)\right\vert \leq \int_{0}^{1}\left\vert 
\frac{\partial k_{1}}{\partial t}(t,s)\right\vert
g_{1}(s)f_{1}(s,u(s),u^{\prime }(s),v(s),v^{\prime }(s))\,ds.
\end{equation*}%
By (\ref{f^rho12}) and (\ref{masti}) and, taking the maximum in $[0,1],$%
\begin{eqnarray*}
\lambda {\rho _{1}} &\leq &\max_{t\in \lbrack 0,1]}\int_{0}^{1}\left\vert 
\frac{\partial k_{i}}{\partial t}(t,s)\right\vert
g_{1}(s)f_{1}(s,u(s),u^{\prime }(s),v(s),v^{\prime }(s))\,ds \\
&\leq &\max_{t\in \lbrack 0,1]}\int_{0}^{1}\left\vert \frac{\partial k_{i}}{%
\partial t}(t,s)\right\vert g_{1}(s){\rho _{1}}f_{1}^{\rho _{1},\rho _{2}}\,ds \\
&\leq &{\rho _{1}}f_{1}^{\rho _{1},\rho _{2}}\dfrac{1}{m_{1}^{\ast }},
\end{eqnarray*}%
we obtain a similar contradiction as above.

The other cases follow the same arguments.
\end{proof}

Secondly, we provide a condition to have a null fixed point index on $K_{\rho_{1},\rho _{2}}$.
\begin{lem}\label{idx0b1} 
Assume that

\begin{enumerate}
\item[$(\mathrm{I}_{\rho _{1},\rho_2}^{0})$] there exist $\rho _{1},\rho _{2}>0$ such that for every $%
i=1,2,$%
\begin{equation}\label{eqMest}
f_{1,(\rho _{1},\rho _{2})}>M_{1},\, f_{1,(\rho _{1},\rho _{2})}^{\ast
}>M_{1}^{\ast },\, f_{2,(\rho _{1},\rho _{2})}>M_{2},\, f_{2,(\rho
_{1},\rho _{2})}^{\ast }>M_{2}^{\ast },  
\end{equation}%
{} where 
\begin{align*}
f_{1,({\rho _{1},\rho _{2}})}& :=\inf \left\{ 
\begin{array}{l}
\dfrac{f_{1}(t,u_{1},u_{2},v_{1},v_{2})}{\rho _{1}}: \\ 
(t,u_{1},u_{2},v_{1},v_{2})\in \lbrack a_{1},b_{1}]\times \lbrack c_{1}\rho
_{1},\rho _{1}]\times \lbrack -\rho _{1},\rho _{1}]\times \lbrack -\rho
_{2},\rho _{2}]^{2}%
\end{array}%
\right\} , \\
f_{1,(\rho _{1},\rho _{2})}^{\ast }& :=\inf \left\{ 
\begin{array}{l}
\dfrac{f_{1}(t,u_{1},u_{2},v_{1},v_{2})}{\rho _{1}}: \\ 
(t,u_{1},u_{2},v_{1},v_{2})\in \lbrack \gamma _{1},\delta_{1}]\times \lbrack
-\rho _{1},\rho _{1}]\times \lbrack d_{1}\rho _{1},\rho _{1}]\times \lbrack
-\rho _{2},\rho _{2}]^{2}%
\end{array}%
\right\} , \\
f_{2,({\rho _{1},\rho _{2}})}& :=\inf \left\{ 
\begin{array}{l}
\dfrac{f_{2}(t,u_{1},u_{2},v_{1},v_{2})}{\rho _{2}}: \\ 
(t,u_{1},u_{2},v_{1},v_{2})\in \lbrack a_{2},b_{2}]\times \lbrack -\rho
_{1},\rho _{1}]^{2}\times \lbrack c_{2}\rho _{2},\rho _{2}]\times \lbrack
-\rho _{2},\rho _{2}]%
\end{array}%
\right\} , \\
f_{2,(\rho _{1},\rho _{2})}^{\ast }& :=\inf \left\{ 
\begin{array}{l}
\dfrac{f_{2}(t,u_{1},u_{2},v_{1},v_{2})}{\rho _{2}}: \\ 
(t,u_{1},u_{2},v_{1},v_{2})\in \lbrack \gamma_{2},\delta _{2}]\times \lbrack
-\rho _{1},\rho _{1}]^{2}\times \lbrack -\rho _{2},\rho _{2}]\times \lbrack
d_{2}\rho _{2},\rho _{2}]%
\end{array}%
\right\} ,
\end{align*}%
and%
\begin{eqnarray}\label{M*_i}
\frac{1}{M_{i}} &:&=\min_{t\in \lbrack
a_{i},b_{i}]}\int_{a_{i}}^{b_{i}}k_{i}(t,s)g_{i}(s)\,ds,\,   \label{M_i}
\\
\frac{1}{M_{i}^{\ast }} &:&=\min_{t\in \lbrack \gamma _{i},\delta_{i}]}
\int_{\gamma_{i}}^{\delta _{i}}\frac{\partial k_{i}}{\partial t}%
(t,s)g_{i}(s)\,ds\,   
\end{eqnarray}
\end{enumerate}

Then $i_{K}(T,K_{\rho _{1},\rho _{2}})=0$.
\end{lem}

\begin{proof}
Consider $e(t)\equiv 1$ for $t\in \lbrack 0,1],$ and note that $(e,e)\in K$.

We claim that 
\begin{equation*}
(u,v)\neq T(u,v)+\lambda (e,e)\quad \text{for }(u,v)\in \partial K_{\rho
_{1},\rho _{2}}\quad \text{and }\lambda \geq 0.
\end{equation*}%
Assume, by contradiction, that there exist $(u,v)\in \partial K_{\rho
_{1},\rho _{2}}$ and $\lambda \geq 0$ such that $(u,v)=T(u,v)+\lambda (e,e)$.

Consider that (\ref{u=rho1}) holds. Then we can assume that for all $t\in
\lbrack a_{1},b_{1}]$ we have 
\begin{align*}
{c_{1}}\rho _{1} \leq u(t)\leq {\rho _{1}},-\rho _{1}\leq u^{\prime
}(t)\leq \rho _{1}, 
-\rho _{2} \leq v(t)\leq {\rho _{2}}\text{ and }-\rho _{2}\leq v^{\prime
}(t)\leq {\rho _{2}}.
\end{align*}%
Then, for $t\in \lbrack a_{1},b_{1}]$, we obtain, by~\eqref{eqMest}, 
\begin{eqnarray*}
u(t) &=&\int_{0}^{1}k_{1}(t,s)g_{1}(s)f_{1}(s,u(s),u^{\prime
}(s),v(s),v^{\prime }(s))\,ds+\lambda e(t) \\
&\geq &\int_{a_{1}}^{b_{1}}k_{1}(t,s)g_{1}(s)f_{1}(s,u(s),u^{\prime
}(s),v(s),v^{\prime }(s))\,ds+{\lambda } \\
&\geq &\int_{a_{1}}^{b_{1}}k_{1}(t,s)g_{1}(s)\rho _{1}f_{1,({\rho _{1},\rho _{2}})}\,ds+%
{\lambda }.
\end{eqnarray*}%
Taking the maximum over $[a_{1},b_{1}]$ gives 
\begin{equation*}
\rho _{1}\geq \max_{t\in \lbrack a_{1},b_{1}]}u(t)\geq {\rho _{1}}f_{1,(\rho
_{1},\rho _{2})}\frac{1}{M_{1}}+{\lambda }.
\end{equation*}%
By \eqref{eqMest}, we obtain the following contradiction: $\rho _{1}>\rho
_{1}+\lambda $.

Suppose that 
\begin{equation*}
-\rho _{1}\leq u(t)\leq \rho _{1},\max_{t\in \lbrack 0,1]}u^{\prime
}(t)=\rho _{1},-\rho _{2}\leq v(t)\leq \rho _{2},-\rho _{2}\leq v^{\prime
}(t)\leq \rho _{2},
\end{equation*}%
holds. Then, that for all $t\in \lbrack \gamma _{1},\delta_{1}]$ we have 
\begin{eqnarray*}
u^{\prime }(t) &=&\int_{0}^{1}\frac{\partial k_{1}}{\partial t}%
(t,s)g_{1}(s)f_{1}(s,u(s),u^{\prime }(s),v(s),v^{\prime }(s))\,ds+\lambda e(t)
\\
&\geq &\int_{\gamma _{1}}^{\delta_{1}}\frac{\partial k_{1}}{\partial t}%
(t,s)g_{1}(s)f_{1}(s,u(s),u^{\prime }(s),v(s),v^{\prime }(s))\,ds+{\lambda } \\
&\geq &\int_{\gamma _{1}}^{\delta_{1}}\frac{\partial k_{1}}{\partial t}%
(t,s)g_{1}(s)\rho _{1}f_{1,(\rho _{1},\rho _{2})}^{\ast }ds+{\lambda }.
\end{eqnarray*}%
Taking the maximum over $[\gamma _{1},\delta_{1}]$ gives 
\begin{equation*}
\rho _{1}\geq \max_{t\in \lbrack \gamma _{1},\delta_{1}]}u^{\prime }(t)\geq {%
\rho _{1}}f_{1,(\rho _{1},\rho _{2})}^{\ast }\frac{1}{M_{1}^{\ast }}+{%
\lambda },
\end{equation*}%
and, by \eqref{M*_i}, a similar contradiction is achieved.

For the other cases the procedure is analogous.
\end{proof}
In the following Theorem we provide a result valid for up to three nontrivial solutions, but 
it is possible to prove the existence of four or more nontrivial solutions; see for 
example~\cite{kljdeds} for the kind of results that may be stated.
We omit the proof that follows, in a routine manner, by means of the properties of fixed point index.

\begin{thm}\label{mult-sys}
The system~\eqref{syst} has at least one nontrivial solution
in $K$ if one of the following conditions holds.

\begin{enumerate}

\item[$(S_{1})$]  For $i=1,2$ there exist $\rho _{i},r _{i}\in (0,\infty )$ with $\rho
_{i}/c_i<r _{i}$ such that 
$(\mathrm{I}_{\rho _{1},\rho_2}^{0})$, 
$(\mathrm{I}_{r _{1},r_2}^{1})$ hold.

\item[$(S_{2})$] For $i=1,2$ there exist $\rho _{i},r _{i}\in (0,\infty )$ with $\rho
_{i}<r _{i}$ such that 
$(\mathrm{I}_{\rho _{1},\rho_2}^{1}),$
$(\mathrm{I}_{r _{1},r_2}^{0})$ hold.
\end{enumerate}

The system \eqref{syst} has at least two nontrivial solutions in $K$ if one of the following conditions holds.

\begin{enumerate}

\item[$(S_{3})$] For $i=1,2$ there exist $\rho _{i},r _{i},s_i\in (0,\infty )$
with $\rho _{i}/c_i<r_i <s _{i}$ such that 
$(\mathrm{I}_{\rho_{1},\rho_2}^{0})$, 
$(\mathrm{I}_{r _{1},r_2}^{1})$ 
\text{and}
$(\mathrm{I}_{s _{1},s_2}^{0})$
hold.

\item[$(S_{4})$] For $i=1,2$ there exist $\rho _{i},r _{i},s_i\in (0,\infty )$
with $\rho _{i}<r _{i}$ and $r _{i}/c_i<s _{i}$ such that 
$(\mathrm{I}_{\rho _{1},\rho_2}^{1}),$
$(\mathrm{I}_{r _{1},r_2}^{0})$ and
$(\mathrm{I}_{s _{1},s_2}^{1})$ hold.
\end{enumerate}

The system \eqref{syst} has at least three nontrivial solutions in $K$ if one
of the following conditions holds.

\begin{enumerate}
\item[$(S_{5})$] For $i=1,2$ there exist $\rho _{i},r _{i},s_i,\sigma_i\in
(0,\infty )$ with $\rho _{i}/c_i<r _{i}<s _{i}$ and $s _{i}/c_i<\sigma
_{i}$ such that 
$(\mathrm{I}_{\rho _{1},\rho_2}^{0}),$ 
$(\mathrm{I}_{r _{1},r_2}^{1}),$ 
$(\mathrm{I}_{s_1,s_2}^{0})$ and 
$(\mathrm{I}_{\sigma _{1},\sigma_2}^{1})$ hold.

\item[$(S_{6})$] For $i=1,2$ there exist 
$\rho _{i},r _{i},s_i,\sigma_i\in (0,\infty )$ with 
$\rho _{i}<r _{i}$ and $r _{i}/c_i<s _{i}<\sigma _{i}$
such that $(\mathrm{I}_{\rho _{1},\rho_2}^{1}),$ 
$(\mathrm{I}_{r_{1},r_2}^{0}),$ $(\mathrm{I}_{s _{1},s_2}^{1})$ 
and $(\mathrm{I}_{\sigma _{1},\sigma_2}^{0})$ hold.
\end{enumerate}
\end{thm}
In the next example we illustrate the applicability of Theorem~\ref{mult-sys}.
\begin{ex}
Consider the system
\begin{equation}\label{exsystch}
 \left\{
\begin{array}{c}
u(t)=\int_{0}^{1}s(\frac{7}{8}t-t^2)\left( \left( u(t)\right) ^{2}+\left(
u^{\prime }(t)\right) ^{2}\right) \left( 2+\cos \left( v(t)\, v^{\prime
}(t)\right) \right)\,ds, \\ 
v(t)=\int_{0}^{1}s(\frac{11}{10}t-t^2-\frac{1}{10})\left( \left( v(t)\right) ^{2}+\left(
v^{\prime }(t)\right) ^{2}\right) \left( 2-\sin \left( u(t)\, u^{\prime
}(t)\right) \right)\,ds.%
\end{array}
 \right.
\end{equation}%
In this case we have
\begin{align*}
&k_1(t,s)=s(\frac{7}{8}t-t^2), k_2(t,s)=s(\frac{11}{10}t-t^2-\frac{1}{10}),\\
&\frac{\partial k_{1}}{\partial t}(t,s)=s(\frac{7}{8}-2t), \frac{\partial k_{2}}{\partial t}(t,s)=s(\frac{11}{10}-2t),\\
&g_{1}(t) \equiv 1,\, g_{2}(t)\equiv 1,\\
&f_{1}(t,u_{1},u_{2},v_{1},v_{2}) =\left( \left( u_{1}\right)
^{2}+\left( u_{2}\right) ^{2}\right) \left( 2+\cos \left( v_{1}\, %
v_{2}\right) \right), \\
&f_{2}(t,u_{1},u_{2},v_{1},v_{2}) =\left( \left( v_{1}\right)
^{2}+\left( v_{2}\right) ^{2}\right) \left( 2-\sin \left( u_{1}\, %
u_{2}\right) \right).
\end{align*}
Note that $k_1$, $k_2$, $\dfrac{\partial k_{1}}{\partial t}$ and  $\dfrac{\partial k_{2}}{\partial t}$ change sign on $[0,1]^2$. The assumption
$(A3)$ is satisfied with the choices
\begin{align*}
\phi _{1}(s) =\frac{49}{256}s ,&\, \phi _{2}(s)=\frac{81}{400}s, \\
a_{1} =\frac{7}{32},\, b_{1}=\frac{21}{32},\, c_{1} =\frac{3}{4},&\ a_{2} =\frac{13}{40},\, b_{2}=\frac{31}{40},\, c_{2}=\frac{3}{4}\\
\psi _{1}(s) =\frac{9}{8}s ,&\, \psi _{2}(s)= \frac{11}{10}s,\\
\gamma _{1}=0,\, \delta _{1}=\frac{7}{32},\, d_{1}=\frac{7}{18},&\, \gamma _{2}=0,\, \delta _{2}=\frac{11}{40},\, d_{2}=\frac{13}{44}, \\
\end{align*}
Furthermore $(A4)$ is satisfied since
$$
\int_{\frac{7}{32}}^{\frac{21}{32}}\frac{49}{256}s\,ds=\frac{2401}{65536},\,
\int_{\frac{13}{40}}^{\frac{31}{40}}\frac{81}{400}s\,ds=\frac{8019}{160000},\,
\int_{0}^{\frac{7}{32}}\frac{9}{8}s\,ds=\frac{441}{16384},\,
\int_{0}^{\frac{11}{40}}\frac{11}{10}s\,ds=\frac{1331}{32000}.
$$
By direct calculation we have
$$
\frac{1}{m_{1}}=\max_{t\in \lbrack 0,1]}\int_{0}^{1} \left\vert s(\frac{7}{8}t-t^2)\right\vert \,ds=\frac{49}{512},\,
\frac{1}{m_{2}}=\max_{t\in \lbrack 0,1]}\int_{0}^{1} \left\vert s(\frac{11}{10}t-t^2-\frac{1}{10})\right\vert \,ds=\frac{81}{800},
$$
$$
\frac{1}{m_{1}^{\ast }}=\max_{t\in \lbrack 0,1]}\int_{0}^{1}\left\vert  s(\frac{7}{8}-2t)\right\vert \,ds=\frac{9}{16},\, 
\frac{1}{m_{2}^{\ast }}=\max_{t\in \lbrack 0,1]}\int_{0}^{1}\left\vert  s(\frac{11}{10}-2t)\right\vert \,ds=\frac{11}{20},
$$
$$
\frac{1}{M_{1}} =\min_{t\in \lbrack
\frac{7}{32}, \frac{21}{32}]}\int_{\frac{7}{32}}^{\frac{21}{32}}s(\frac{7}{8}t-t^2)\,ds=\frac{7203}{262144},\, 
\frac{1}{M_{2}} =\min_{t\in \lbrack
\frac{13}{40}, \frac{31}{40}]}\int_{\frac{13}{40}}^{\frac{31}{40}}s(\frac{11}{10}t-t^2-\frac{1}{10})\,ds=\frac{24057}{640000},\, 
$$
$$
\frac{1}{M_{1}^{\ast }} =\min_{t\in \lbrack {0},{\frac{7}{32}}]}
\int_{0}^{\frac{7}{32}}s(\frac{7}{8}-2t)\,ds=\frac{343}{32768},\, 
\frac{1}{M_{2}^{\ast }} =\min_{t\in \lbrack 0, \frac{11}{40}]}
\int_{0}^{\frac{11}{40}}s(\frac{11}{10}-2t)\,ds
=\frac{1331}{64000}. 
$$
Now we need
$$
f_{1}^{\rho _{1},\rho _{2}}\leq 6 \rho _{1}<\min \left\{ m_{1},m_{1}^{\ast }\right\} =\frac{16}{9}\quad (\text{true if}\ \rho_1<\frac{1}{27}) ,
$$
and
$$
f_{2}^{\rho _{1},\rho _{2}}\leq 6 \rho _{2}<\min \left\{ m_{2},m_{2}^{\ast }\right\} =\frac{20}{11} \quad (\text{true if}\ \rho_2<\frac{10}{33}).
$$
Furthermore we need 
\begin{align*}
f_{1,({\rho _{1},\rho _{2}})}& \geq \frac{9}{16}\rho_{1} >M_{1}=\frac{262144}{7203}\quad (\text{true if}\ \rho_1>\frac{4194304}{64827}) , \\
f_{1,(\rho _{1},\rho _{2})}^{\ast }& \geq \frac{49}{324}\rho_{1} >M_{1}^{\ast }=\frac{32768}{343}\quad (\text{valid if}\ \rho_1> \frac{10616832}{16807}) , \\
f_{2,({\rho _{1},\rho _{2}})}& \geq \frac{9}{16}\rho_{2} >M_{2}=\frac{640000}{24057}\quad (\text{true if}\ \rho_2>\frac{10240000}{216513}) , \\
f_{2,(\rho _{1},\rho _{2})}^{\ast }& \geq \frac{169}{1936}\rho_{2} >M_{2}^{\ast }=\frac{64000}{1331}\quad (\text{true if}\ \rho_2>\frac{1024000}{1859}) .
\end{align*}%
Thus if we fix 
$$0<\rho_1<\frac{1}{27},\, 0<\rho_2<\frac{10}{33},$$
$$r_1>\max\{\frac{4194304}{64827},\frac{10616832}{16807}\}=\frac{10616832}{16807},\,
r_2>\max\{\frac{10240000}{216513},\frac{1024000}{1859}\}=\frac{1024000}{1859}$$ 
the conditions $(\mathrm{I}_{\rho _{1},\rho_2}^{1}),$
$(\mathrm{I}_{r _{1},r_2}^{0})$ hold and we obtain, by Theorem~\ref{mult-sys}, the existence of one nontrivial solution of the system~ \eqref{exsystch}.
\end{ex}
\begin{rem}
Note that in the case of \emph{non-negative} kernels, the same reasoning as above provides the existence of \emph{positive} solutions.
In this case one may use the smaller cones (with abuse of notation)
\begin{equation*}
\tilde{K_{i}}:=\left\{ w\in C^{1}[0,\,1]: w\geq 0, \underset{t\in \lbrack a_{i},b_{i}]}%
{\min }w(t)\geq c_{i}\Vert w\Vert _{C},\underset{t\in \lbrack \gamma_{i}, \delta _{i}]}
{\min }w^{\prime }(t)\geq d_{i}\Vert w^{\prime }\Vert
_{C}\right\}.
\end{equation*}%

If, additionally, the derivative with respect to $t$ of the kernels is non-negative, one may seek solutions in the even smaller cone (again with abuse of notation) given by 
\begin{equation*}
\tilde{K_{i}}:=\left\{ w\in C^{1}[0,\,1]: w\geq 0,  w'\geq 0, \underset{t\in \lbrack a_{i},b_{i}]}%
{\min }w(t)\geq c_{i}\Vert w\Vert _{C},\underset{t\in \lbrack \gamma_{i},\delta_{i}]}
{\min }w^{\prime }(t)\geq d_{i}\Vert w^{\prime }\Vert
_{C}\right\}.
\end{equation*}%

For brevity we do not re-state all the results within these frameworks, but we illustrate the latter situation in Section~\ref{appsys}, when discussing the system~\eqref{bvp-ex-intro}.
\end{rem}
We now give sufficient conditions for the non-existence of nontrivial solutions for the system~\eqref{syst}.

\begin{thm}
Let $m_{i}$ be given by (\ref{mi}), $M_{i}$ be given by (\ref{M_i}) and $a_{i},b_{i},c_{i}$ as in (A3) and
suppose that the following conditions $(N1)$ and $(N2)$ are satisfied:
\begin{itemize}
\item[$(N1)$] Either
\begin{equation}\label{cond1}
f_{1}(t,u_{1},u_{2},v_{1},v_{2})<m_{1}|u_{1}|\ \text{for every}\ t\in
\lbrack 0,1],u_{1}\neq 0 \text{ and }\ u_{2},v_{1},v_{2}\in \mathbb{R};  
\end{equation}%
or 
\begin{equation}\label{cond2}
f_{1}(t,u_{1},u_{2},v_{1},v_{2})>\frac{M_{1}}{c_{1}}u_{1}\ \text{for every}\
t\in \lbrack a_{1},b_{1}],\, u_{1}>0\text{ and }u_{2},v_{1},v_{2}\in 
\mathbb{R},
\end{equation}
holds.
{}
\item[$(N2)$] Either
\begin{equation*}
f_{2}(t,u_{1},u_{2},v_{1},v_{2})<m_{2}|v_{1}|\ \text{for every}\ t\in
\lbrack 0,1],v_{1}\neq 0\text{ and }u_{1},u_{2},v_{2}\in 
\mathbb{R}
\,;
\end{equation*}%
or
\begin{equation*}
f_{2}(t,u_{1},u_{2},v_{1},v_{2})>\frac{M_{2}}{c_{2}}v_{1}\ \text{for every}\
t\in \lbrack a_{2},b_{2}],\, v_{1}>0\text{ and } u_{1},u_{2},v_{2}\in 
\mathbb{R}
,
\end{equation*}%
holds.
\end{itemize}
Then there is no nontrivial solution of the system \eqref{syst} in the cone $K$ given by~\eqref{cone-sign}.
\end{thm}

\begin{proof}
Suppose, by contradiction, that there exists a nontrivial solution of \eqref{syst} in $K$, that is, $(u,v)\in K$ such that $(u,v)=T(u,v)$ and $(u,v)\neq (0,0)$.
Assume, without loss of generality, that $\|u\|_{C}\neq 0$. If \eqref{cond1} holds, then, for $t\in \lbrack 0,1]$, we have
\begin{align*}
|u(t)|& \leq \int_{0}^{1}|k_{1}(t,s)|g_{1}(s)f_{1}(s,u(s),u^{\prime
}(s),v(s),v^{\prime }(s))\,ds \\
& <m_{1}\int_{0}^{1}|k_{1}(t,s)|g_{1}(s)|u(s)|\,ds\leq m_{1}\Vert u\Vert
_{C}\int_{0}^{1}|k_{1}(t,s)|g_{1}(s)\,ds.
\end{align*}%
Taking the maximum for $t\in \lbrack 0,1]$, we have, by \eqref{mi}, the
following contradiction 
\begin{equation*}
\Vert u\Vert _{C}<m_{1}\Vert u\Vert _{C}\sup_{t\in \lbrack
0,1]}\int_{0}^{1}|k_{1}(t,s)|g_{1}(s)\,ds=\Vert u\Vert _{C}.
\end{equation*}
If  \eqref{cond2} holds, then, for $t\in \lbrack a_{1},b_{1}]$, we have 
\begin{align*}
u(t)& =\int_{0}^{1}k_{1}(t,s)g_{1}(s)f_{1}(s,u(s),u^{\prime
}(s),v(s),v^{\prime }(s))\,ds \\
& >\int_{a_{1}}^{b_{1}}k_{1}(t,s)g_{1}(s)\frac{M_{1}}{c_{1}}u(s)\,ds.
\end{align*}%
Taking the minimum for $t\in \lbrack a_{1},b_{1}]$, we obtain, for some $%
\xi _{1}>0,$ the following contradiction, by (\ref{M_i}) and (\ref{Ktil}), 
\begin{align*}
\xi _{1} =&\min_{t\in \lbrack a_{1},b_{1}]}u(t)>\frac{M_{1}}{c_{1}}%
\inf_{t\in \lbrack
a_{1},b_{1}]}\int_{a_{1}}^{b_{1}}k_{1}(t,s)g_{1}(s)\min_{s\in \lbrack
a_{1},b_{1}]}u(s)\,ds \\
\geq &M_{1}\Vert u\Vert _{C}\inf_{t\in \lbrack
a_{1},b_{1}]}\int_{a_{1}}^{b_{1}}k_{1}(t,s)g_{1}(s)\,ds=\Vert u\Vert _{C}\geq
\xi _{1}.
\end{align*}%
The proof in the case of $\|v\|_{C}\neq 0$ follows as above, using the condition $(N2)$.
\end{proof}

\section{Positive solutions of some third order systems}\label{appsys}
We turn back our attention to the system of third order ODEs with three point boundary conditions
\begin{equation}
\left\{ 
\begin{array}{c}
-u^{\prime \prime \prime }(t)=g_{1}(t)f_{1}(t,\,u(t),\,u^{\prime
}(t),v(t),\,v^{\prime }(t)), \\ 
-v^{\prime \prime \prime }(t)=g_{2}(t)f_{2}(t,\,u(t),\,u^{\prime
}(t),v(t),\,v^{\prime }(t)), \\ 
u(0)=u^{\prime }(0)=0,u^{\prime }(1)=\alpha_{1} u^{\prime }(\eta_{1} ), \\ 
v(0)=v^{\prime }(0)=0,v^{\prime }(1)=\alpha_{2} v^{\prime }(\eta_{2} ),%
\end{array}%
\right.  \label{eq1}
\end{equation}%
where, for $i=1,2,$ 
$f_{i}: [0,\,1]\times \lbrack 0,\,+\infty
)^{4}\to [0,\,+\infty )$ is a $L^{\infty }$-Carath\'{e}odory function, $g_{i} \in L^{1}[0,1]$ with $g_{i}(t)\geq 0$ for a.e. $t\in \lbrack 0,\,1],$ $%
0<\eta_{i} <1$ and $%
1<\alpha_{i}  <\frac{1}{\eta_{i} }$.

By routine calculation we can associate to the system~\eqref{eq1} the system of Hammerstein integral equations
\begin{equation}
\left\{ 
\begin{array}{c}
u(t)=\int_{0}^{1}k_{1}(t,s)g_1(s)f_{1}(s,\,u(s),\,u^{\prime }(s),v(s),\,v^{\prime
}(s))\,ds, \\ 
v(t)=\int_{0}^{1}k_{2}(t,s)g_2(s)f_{2}(s,\,u(s),\,u^{\prime }(s),v(s),\,v^{\prime
}(s))\,ds,%
\end{array}%
\right.  \label{eq2}
\end{equation}
where $k_{i}(t,s)$ are the Green's function given by
\begin{equation}\label{Green1}
k_{i}(t,s)=\frac{1}{2(1-\alpha \eta_{i} )}\left\{ 
\begin{array}{cc}
(2ts-s^{2})(1-\alpha_{i} \eta_{i} )+t^{2}s(\alpha_{i}  -1), & s\leq \min \{\eta_{i} ,t\}, \\ 
t^{2}(1-\alpha_{i} \eta_{i} )+t^{2}s(\alpha_{i}  -1), & t\leq s\leq \eta_{i} , \\ 
(2ts-s^{2})(1-\alpha_{i}  \eta_{i} )+t^{2}(\alpha_{i}  \eta_{i}  -s), & \eta_{i} \leq s\leq t, \\ 
t^{2}(1-s), & \max \{\eta_{i} ,t\}\leq s.%
\end{array}%
\right.  
\end{equation}%
The derivatives of the Green's functions~\eqref{Green1} are given by
\begin{equation}\label{Greender1}
\frac{\partial k_{i}}{\partial t}(t,s)=\frac{1}{(1-\alpha_{i}  \eta_{i} )}\left\{ 
\begin{array}{cc}
s(1-\alpha_{i}  \eta_{i} )+ts(\alpha_{i}  -1), & s\leq \min \{\eta_{i} ,t\}, \\ 
t(1-\alpha_{i}  \eta_{i} )+ts(\alpha_{i}  -1), & t\leq s\leq \eta_{i} , \\ 
s(1-\alpha_{i}  \eta_{i} )+t(\alpha_{i}  \eta_{i} -s), & \eta_{i} \leq s\leq t, \\ 
t(1-s), & \max \{\eta_{i} ,t\}\leq s,%
\end{array}%
\right.
\end{equation}%

The following Lemmas provide some useful properties of the Green's functions and their
derivatives.

\begin{lem}[\cite{Li-jun}]\label{Lema phi1}
Take $0<\eta_{i} <1$, $1<\alpha_{i}  <\frac{1}{\eta_{i} 
}$ and $k_{i}$ as in~\eqref{Green1}. Then we have  
$$%
0\leq k_{i}(t,s)\leq \phi _{i}(s),\ (t,s)\in \lbrack 0,\,1]\times \lbrack 0,\,1],
$$ where 
\begin{equation*}
\phi _{i}(s)=\frac{1+\alpha_{i}  }{1-\alpha_{i}  \eta_{i} }s(1-s).
\end{equation*}
Furthermore we have 
$$
k_{i}(t,s)\geq
c_{i}\phi _{i}(s),\ \ (t,s)\in [ \frac{\eta_{i} }{\alpha_{i}  },\,\eta_{i} ]\times
\lbrack 0,\,1],
$$ where 
\begin{equation}\label{c1}
0<c_{i}=\frac{\eta_{i} ^{2}}{2\alpha_{i}^{2}(1+\alpha_{i}  )}\min \{\alpha_{i} -1,\,1\}<1.
\end{equation}
\end{lem}

\begin{lem}[\cite{fm+rs}]\label{Lema psi1}
Take $0<\eta_{i} <1,$ $1<\alpha_{i}  <\frac{1}{\eta_{i} }$,  $\frac{\partial k_{i}}{\partial t}$ as in~\eqref{Greender1}. Then we have
$$
0\leq \frac{
\partial k_{i}}{\partial t}(t,s)\leq \psi _{i}(s),\  (t,s)\in \lbrack 0,\,1]\times \lbrack 0,\,1],
$$ 
where 
\begin{equation*}
\psi _{i}(s)=\frac{(1-s)}{(1-\alpha_{i}  \eta_{i} )}.
\end{equation*}
Furthermore we have
$$
\frac{\partial k_{i}}{\partial t}(t,s)\geq
d_{i}\psi _{i}(s) ,\ (t,s)\in \lbrack \frac{\eta_{i}}{\alpha_{i} },\,\eta_{i} ]\times \lbrack 0,\,1],
$$ 
with 
\begin{equation}
0<d_{i}=\min \{\alpha_{i}  \eta_{i} ,\;\eta_{i} \}<1.
\end{equation}
\end{lem}

From Lemmas~\ref{Lema phi1} and~\ref{Lema psi1} we obtain that $k_{i}$ satisfies a stronger positivity requirement than $(A3)$. This setting  
 enables us to work in the cone
\begin{equation}\label{cone-pos}
\begin{array}{c}
K:=\{(u,v)\in \tilde{K_{1}}\times \tilde{K_{2}}\},%
\end{array}%
\end{equation}
where
\begin{equation*}
\tilde{K_{i}}:=\left\{ w\in C^{1}[0,\,1]: w\geq 0,  w'\geq 0, \underset{t\in [ \frac{\eta_{i} }{\alpha_{i}  },\,\eta_{i} ]}%
{\min }w(t)\geq c_{i}\Vert w\Vert _{C},\underset{t\in [ \frac{\eta_{i} }{\alpha_{i}  },\,\eta_{i} ]}{\min }w^{\prime }(t)\geq d_{i}\Vert w^{\prime }\Vert
_{C}\right\}.
\end{equation*}%
The condition $(\mathrm{I}_{\rho _{1},\rho_2}^{1})$  in this case reads as follows.
\begin{itemize}
\item[$(\mathrm{I}_{\rho _{1},\rho_2}^{1})$] 
 there exist $\rho _{1},\rho _{2}>0$ such that for every $%
i=1,2,$ 
$
f_{i}^{\rho _{1},\rho _{2}}<\min \left\{ m_{i},m_{i}^{\ast }\right\} 
$
where 
\begin{equation*}
f_{i}^{\rho _{1},\rho _{2}}:=\sup \left\{ \frac{%
f_{i}(t,u_{1},u_{2},v_{1},v_{2})}{\rho _{i}}:(t,u_{1},u_{2},v_{1},v_{2})\in
\lbrack 0,1]\times \lbrack 0,\rho _{1}]^{2}\times \lbrack 0, \rho _{2}]^{2}\right\} ,  
\end{equation*}%
\begin{equation*}
\frac{1}{m_{i}}=\max_{t\in \lbrack 0,1]}\int_{0}^{1}k_{i}(t,s)g_{i}(s)\,ds,\quad  \frac{1}{m_{i}^{\ast }}=\max_{t\in \lbrack 0,1]}\int_{0}^{1} 
\frac{\partial k_{i}}{\partial t}(t,s) g_{i}(s)\,ds.
\end{equation*}%
\end{itemize}

On the other hand, the condition $(\mathrm{I}_{\rho _{1},\rho_2}^{0})$ reads as follows.

\begin{enumerate}
\item[$(\mathrm{I}_{\rho _{1},\rho_2}^{0})$] there exist $\rho _{1},\rho _{2}>0$ such that for every $%
i=1,2,$%
\begin{equation}
f_{1,(\rho _{1},\rho _{2})}>M_{1},\, f_{1,(\rho _{1},\rho _{2})}^{\ast
}>M_{1}^{\ast },\, f_{2,(\rho _{1},\rho _{2})}>M_{2},\, f_{2,(\rho
_{1},\rho _{2})}^{\ast }>M_{2}^{\ast },  
\end{equation}%
{} where 
\begin{align*}
f_{1,({\rho _{1},\rho _{2}})}& :=\inf \left\{ 
\begin{array}{l}
\dfrac{f_{1}(t,u_{1},u_{2},v_{1},v_{2})}{\rho _{1}}: \\ 
(t,u_{1},u_{2},v_{1},v_{2})\in \lbrack a_{1},b_{1}]\times \lbrack c_{1}\rho
_{1},\rho _{1}]\times \lbrack 0,\rho _{1}]\times \lbrack 0,\rho _{2}]^{2}%
\end{array}%
\right\} , \\
f_{1,(\rho _{1},\rho _{2})}^{\ast }& :=\inf \left\{ 
\begin{array}{l}
\dfrac{f_{1}(t,u_{1},u_{2},v_{1},v_{2})}{\rho _{1}}: \\ 
(t,u_{1},u_{2},v_{1},v_{2})\in \lbrack \gamma _{1},\delta_{1}]\times \lbrack
0,\rho _{1}]\times \lbrack d_{1}\rho _{1},\rho _{1}]\times \lbrack
0,\rho _{2}]^{2}%
\end{array}%
\right\} , \\
f_{2,({\rho _{1},\rho _{2}})}& :=\inf \left\{ 
\begin{array}{l}
\dfrac{f_{2}(t,u_{1},u_{2},v_{1},v_{2})}{\rho _{2}}: \\ 
(t,u_{1},u_{2},v_{1},v_{2})\in \lbrack a_{2},b_{2}]\times \lbrack 0,\rho _{1}]^{2}\times \lbrack c_{2}\rho _{2},\rho _{2}]\times \lbrack
0,\rho _{2}]%
\end{array}%
\right\} , \\
f_{2,(\rho _{1},\rho _{2})}^{\ast }& :=\inf \left\{ 
\begin{array}{l}
\dfrac{f_{2}(t,u_{1},u_{2},v_{1},v_{2})}{\rho _{2}}: \\ 
(t,u_{1},u_{2},v_{1},v_{2})\in \lbrack \gamma_{2},\delta _{2}]\times \lbrack
0,\rho _{1}]^{2}\times \lbrack 0,\rho _{2}]\times \lbrack
d_{2}\rho _{2},\rho _{2}]%
\end{array}%
\right\}.
\end{align*}%
\end{enumerate}

We can now state an existence result for one nontrivial solution for the System~\eqref{eq1}. Note that it is possible to state a result for two or more nontrivial solutions, in the spirit of Theorem~\ref{mult-sys}.

\begin{thm}
\label{Thm Green}
For $i=1,2,$ let
$f_{i}: [0,\,1]\times \lbrack 0,\,+\infty
)^{4}\to [0,\,+\infty )$ be a $L^{\infty }$-Carath\'{e}odory function and let $g_{i} \in L^{1}[0,1]$ be such that $g_{i}(t) \geq 0$ for a.e. $t\in[0,1]$ and 
\begin{itemize}
\item[($A^{\ast }$4)]  \begin{equation*}
\int_{\frac{\eta_{i} }{\alpha_{i}  }}^{\eta_{i} }\frac{1+\alpha_{i}  }{1-\alpha_{i}  \eta_{i} }%
s(1-s)g_{i}(s)\,ds >0, \quad
\int_{\frac{\eta_{i} }{\alpha_{i}  }}^{\eta_{i} }\frac{(1-s)}{(1-\alpha_i \eta_{i} )}g_{i}(s)\,ds
>0.
\end{equation*}%
\end{itemize}
The system~\eqref{eq1} admits a nontrivial solution with non-negative, non-decreasing components if one of the following conditions hold.
\begin{itemize}
\item[$(\hat{S}_{1})$]  For $i=1,2$ there exist $\rho _{i},r _{i}\in (0,\infty )$ with $\rho
_{i}/c_i<r _{i}$ such that 
$(\mathrm{I}_{\rho _{1},\rho_2}^{0})$, 
$(\mathrm{I}_{r _{1},r_2}^{1})$ hold.

\item[$(\hat{S}_{2})$] For $i=1,2$ there exist $\rho _{i},r _{i}\in (0,\infty )$ with $\rho
_{i}<r _{i}$ such that 
$(\mathrm{I}_{\rho _{1},\rho_2}^{1}),$
$(\mathrm{I}_{r _{1},r_2}^{0})$ hold.
\end{itemize}
\end{thm}

\begin{ex}

Consider the following third order nonlinear system%
\begin{equation}\label{EX}
\left\{ 
\begin{array}{c}
-u^{\prime \prime \prime }(t)=t\left( \left( u(t)\right) ^{2}+\left(
u^{\prime }(t)\right) ^{2}\right) \left( 2+\cos \left( v(t)\, v^{\prime
}(t)\right) \right), \\ 
-v^{\prime \prime \prime }(t)=t\left( \left( v(t)\right) ^{2}+\left(
v^{\prime }(t)\right) ^{2}\right) \left( 2-\sin \left( u(t)\, u^{\prime
}(t)\right) \right), \\ 
u(0)=u^{\prime }(0)=0,u^{\prime }(1)=\frac{3}{2}u^{\prime }\left( \frac{1}{2}%
\right), \\ 
v(0)=v^{\prime }(0)=0,v^{\prime }(1)=2v^{\prime }\left( \frac{1}{3}\right) .%
\end{array}%
\right.  
\end{equation}
The system~\eqref{EX} is a particular case of the system~\eqref{eq1} with%
\begin{align*}
g_{1}(t) &\equiv 1,\, g_{2}(t)\equiv 1, \\
f_{1}(t,u_{1},u_{2},v_{1},v_{2}) =&t\left( \left( u_{1}\right)
^{2}+\left( u_{2}\right) ^{2}\right) \left( 2+\cos \left( v_{1}\, %
v_{2}\right) \right), \\
f_{2}(t,u_{1},u_{2},v_{1},v_{2}) =&t\left( \left( v_{1}\right)
^{2}+\left( v_{2}\right) ^{2}\right) \left( 2-\sin \left( u_{1}\, %
u_{2}\right) \right), \\
\eta_{1} =&\frac{1}{2},\text{\ }\alpha_{1}  =\frac{3}{2},\, \eta_{2} =\frac{1}{3},%
\text{\ }\alpha_{2} =2.
\end{align*}
Note that $f_{1}$ and $f_{2}$ are continuous and non-negative. 

Furthermore we may take
\begin{align*}
\phi _{1}(s) =10s\left( 1-s\right) ,&\, \phi _{2}(s)=9s\left(
1-s\right) , \\
\psi _{1}(s) =4\left( 1-s\right) ,&\, \psi _{2}(s)=3\left( 1-s\right) ,
\\
c_{1} =\frac{1}{45},\, c_{2}=\frac{1}{216},&\, d_{1}=\frac{1}{2},%
\, d_{2}=\frac{1}{3}, \\
a_{1} =\gamma _{1}=\frac{1}{3},&\, b_{1}=\delta _{1}=\frac{1}{2},\text{ 
} \\
a_{2} =\gamma_{2}=\frac{1}{6},&\, b_{2}=\delta _{2}=\frac{1}{3}.
\end{align*}

Moreover, as 
\begin{align*}
\int_{\frac{1}{3}}^{\frac{1}{2}}10s\left( 1-s\right) \,ds =\frac{65}{162},&%
\, \int_{\frac{1}{6}}^{\frac{1}{3}}9s\left( 1-s\right) \,ds=\frac{5}{18},%
\\ 
\int_{\frac{1}{3}}^{\frac{1}{2}}4\left( 1-s\right) \,ds =\frac{7}{18},&%
\, \int_{\frac{1}{6}}^{\frac{1}{3}}3\left( 1-s\right) \,ds=\frac{3}{8},
\end{align*}
assumption $(A^{\ast }4)$ holds$.$

We have%
\begin{align*}
\frac{1}{m_{1}} =&\max_{t\in \lbrack 0,1]}\int_{0}^{1}k_{1}(t,s)\,ds \\
\leq & \max_{t\in \lbrack 0,1]}\left( 
\begin{array}{c}
\int_{0}^{\frac{1}{2}}\left( t^{2}s+ts-s^{2}\right) \,ds+\int_{\frac{1}{2}%
}^{1-\frac{\sqrt{2}}{2}t}\left( -2t^{2}s+ts+\frac{3}{2}t^{2}-s^{2}\right) ds
\\ 
+\int_{1-\frac{\sqrt{2}}{2}t}^{1}\left( -2t^{2}s+2t^{2}-s^{2}\right) ds%
\end{array}%
\right) 
=\frac{1}{24}+\frac{\sqrt{2}}{3},\\
\frac{1}{m_{1}^{\ast }} =&\max_{t\in \lbrack 0,1]}\int_{0}^{1}\frac{%
\partial k_{1}}{\partial t}(t,s)g_{1}(s)\,ds 
\leq \max_{t\in \lbrack 0,1]}\left( \int_{0}^{\frac{1}{2}}2ts+s\,ds+\int_{%
\frac{1}{2}}^{1}-4ts+3t+s\, ds\right) =\frac{3}{4},\\
\frac{1}{m_{2}} =&\max_{t\in \lbrack 0,1]}\int_{0}^{1}k_{2}(t,s)\,ds \\
\leq & \max_{t\in \lbrack 0,1]}\left( \int_{0}^{\frac{1}{3}}\frac{3}{2}%
t^{2}s+ts-\frac{s^{2}}{2}\,ds+\int_{\frac{1}{3}}^{1}-\frac{3}{2}t^{2}s+ts-%
\frac{s^{2}}{2}+3t^{2}ds\right) 
=\frac{43}{324},
\end{align*}%
\begin{multline*}
\frac{1}{m_{2}^{\ast }} =\max_{t\in \lbrack 0,1]}\int_{0}^{1}\frac{%
\partial k_{2}}{\partial t}(t,s)\,ds 
\leq \max_{t\in \lbrack 0,1]}\left( \int_{0}^{\frac{1}{3}}3ts+s\,ds+\int_{%
\frac{1}{3}}^{1}-3ts+s+6t\,ds\right) =\frac{10}{3},
\end{multline*}%
\begin{equation*}
\frac{1}{M_{1}}=\min_{t\in \lbrack \frac{1}{3},\frac{1}{2}]}\int_{\frac{1}{3}%
}^{\frac{1}{2}}k_{1}(t,s)\,ds=\min_{t\in \lbrack \frac{1}{3},\frac{1}{2}%
]}\int_{\frac{1}{3}}^{\frac{1}{2}}\left( t^{2}s+\frac{t^{2}}{2}\right) ds=%
\frac{11}{648},
\end{equation*}%
\begin{equation*}
\frac{1}{M_{1}^{\ast }}=\min_{t\in \lbrack \frac{1}{3},\frac{1}{2}]}\int_{%
\frac{1}{3}}^{\frac{1}{2}}\frac{\partial k_{1}}{\partial t}%
(t,s)\,ds=\min_{t\in \lbrack \frac{1}{3},\frac{1}{2}]}\int_{\frac{1}{3}}^{%
\frac{1}{2}}\left( 2ts+t\right) ds=\frac{11}{108},
\end{equation*}%
\begin{equation*}
\frac{1}{M_{2}}=\min_{t\in \lbrack \frac{1}{6},\frac{1}{3}]}\int_{\frac{1}{6}%
}^{\frac{1}{3}}k_{2}(t,s)\,ds\,=\min_{t\in \lbrack \frac{1}{6},\frac{1}{3}%
]}\int_{\frac{1}{6}}^{\frac{1}{3}}\left( \frac{3}{2}t^{2}s+ts-\frac{s^{2}}{2}%
\right) ds=\frac{17}{5184},
\end{equation*}%
\begin{equation*}
\frac{1}{M_{2}^{\ast }}=\min_{t\in \lbrack \frac{1}{6},\frac{1}{3}]}\int_{%
\frac{1}{6}}^{\frac{1}{3}}\frac{\partial k_{2}}{\partial t}%
(t,s)\,ds=\min_{t\in \lbrack \frac{1}{6},\frac{1}{3}]}\int_{\frac{1}{6}}^{%
\frac{1}{3}}\left( 3ts+t\right) ds=\frac{7}{144},
\end{equation*}%
and therefore we obtain
\begin{align*}
m_{1} =&\frac{1}{\frac{1}{24}+\frac{\sqrt{2}}{3}},\, m_{1}^{\ast }=%
\frac{4}{3},\, m_{2}=\frac{324}{43},\, m_{2}^{\ast }=\frac{3}{10},
\\
M_{1} =&\frac{648}{11},\, M_{1}^{\ast }=\frac{108}{11},\, M_{2}=%
\frac{5184}{17},\, M_{2}^{\ast }=\frac{144}{7}.
\end{align*}

Moreover, for 
\begin{equation*}
\rho _{1}<\frac{2}{9}\text{ and }\rho _{2}<\frac{1}{20}
\end{equation*}%
we obtain%
\begin{align*}
f_{1}^{\rho _{1},\rho _{2}} \leq 6\rho _{1}<\min \left\{ m_{1},m_{1}^{\ast }\right\} =\frac{4}{3},
\end{align*}%
\begin{align*}
f_{2}^{\rho _{1},\rho _{2}} 
\leq 6\rho _{2}<\min \left\{ m_{2},m_{2}^{\ast }\right\} =\frac{3}{10}.
\end{align*}

Taking 
\begin{equation*}
\rho _{1}>\frac{3936\,600}{11}\text{ and }\rho _{2}>\frac{279\,936}{17}
\end{equation*}%
we obtain
\begin{eqnarray*}
f_{1,({\rho }_{1}{,\rho _{2}})} 
>&\frac{\rho _{1}}{6075}>M_{1}=\frac{648}{11}, \\
f_{1,({\rho }_{1}{,\rho _{2}})}^{\ast } 
>&\frac{\rho _{1}}{12}>M_{1}^{\ast }=\frac{108}{11},\\
f_{2,({\rho _{1},\rho _{2}})} 
>&\frac{\rho _{2}}{54}>M_{2}=\frac{5184}{17}, \\
f_{2,({\rho _{1},\rho _{2}})}^{\ast } 
>&\frac{\rho _{2}}{54}>M_{2}^{\ast }=\frac{144}{7},
\end{eqnarray*}%
that is, assumption $(\hat{S}_{2})$ holds. 

Therefore all the assumptions of Theorem~\ref{Thm Green} are satisfied.
\end{ex}
\section*{Acknowledgments}
G. Infante was partially supported by G.N.A.M.P.A. - INdAM (Italy).
F. Minh\'{o}s was supported by National Founds through FCT-Funda\c{c}\~{a}o
para a Ci\^{e}ncia e a Tecnologia, project SFRH/BSAB/114246/2016.
This manuscript was partially written during the authors' visits in the reciprocal institutions. 
G. Infante would like to thank the people of the Departamento de Matem\'{a}tica of the 
Universidade de \'{E}vora for their kind hospitality and financial support.

\end{document}